\documentclass{article}
\usepackage[utf8]{inputenc}
\usepackage{amssymb,amsmath,amsthm,amscd}
\usepackage{enumerate}
\usepackage{graphicx}
\usepackage{color}

\usepackage{lmodern} 

\title{Free Fourier Multipliers  associated with the first Segment}

\author{Tao Mei and Quanhua Xu}

  \newcommand{\N}{\ensuremath{\mathbf{N}}}%

  \newcommand{\M}  {\cal M}%

  \newcommand{\F}{\ensuremath{\Bbb{F}}}%
  
 \newcommand{\la}{\lambda}

\newtheorem{thm}{Theorem}

\newtheorem{lemma}[thm]{Lemma}
\newtheorem{corollary}[thm]{Corollary}

\begin{document}


\maketitle
\begin{abstract} We study  Fourier multipliers on   free group $\F_\infty$  associated with the first segment of the reduced words, and prove that they are  completely bounded  on the noncommutative $L^p$ spaces  $L^p(\hat\F_\infty)$ iff   
their restriction on $L^p(\hat\F_1)=L^p({\Bbb T})$ are completely bounded.
As a consequence, we get an analogue of the classical Mikhlin multiplier theorem  for this class of Fourier multipliers on free groups. 
\end{abstract}
\section*{Notation}

$\F_n:$ free group of rank $n\in \N\cup \{\infty\}$ with given generators $g_k$'s.\\
$\la_g$: the  left translation operator on $\ell_2(\F_n)$ sending $\delta_h$ to $\delta_{gh}$.\\
${\cal L}(\F_n)$: the group von Neumann algebra is the weak $*$-closure of the space of linear combinations of $\la_g$ in $B(\ell_2(\F_n))$.\\
$\tau$: the canonical trace on ${\cal L}(\F_n)$ is the linear functional such that $\tau(\la_e)=1$ and $\tau(\la_h)=0$ if $h\neq e$.\\
Set $L^\infty(\hat \F_n)={\cal L}(\F_n)$ by convention.\\
$L^p(\hat \F_n)$,  $1\leq p<\infty$ : the non commutative $L^p$ space  is the completion of ${\cal L}(\F_n)$ 
with respect to the norm $(\tau|x|^p)^\frac1p$.\\
For a reduced word $h=g_{i_1}^{k_1}g_{i_2}^{k_2}\cdots g_{i_m}^{k_m}$, we denote by $ \|h\|$ the block length $m$.
Let ${\cal L}_0=\{e\}$. Denote by ${\cal L}_{k^\pm}$ the set of all reduced words $h$ that starts with a power of $g_k$, i.e.  $h=g_{i_1}^{k_1}g_{i_2}^{k_2}\cdots g_{i_m}^{k_m}$ with $i_1=k$.\\
Let $L_{k^\pm}$ be the projection on to $\la(L_k)$ in $L_2(\hat\F_2)$.\\
Let $e_{kj}$ be the canonical basis of $B(\ell_2)$.\\
Let $L^p(\ell_{cr}^2), 2\leq p\leq \infty$ be the space of operator valued sequences $x=(x_k)_k $ such that $$\|x\|_{L^p(\ell_{cr}^2)}=\max\{\|(\sum_k|x_k|^2)^\frac12\|_p,\|(\sum_k|x^*_k|^2)^\frac12\|_p\}<\infty.$$

\section{Introduction}
 
 The Fourier multiplier operators  have been a central object in analysis. Their boundedness on   $L^p$-spaces of ${\Bbb R}^n$    has been extensively studied. 
The so-called Mikhlin multiplier theorem  says that, if $m$ is a smooth function on ${\Bbb R}^n$  such that  $$\sup_{0\leq |j|\leq \frac n2+1} |\xi^j\nabla^j m(\xi)|<C$$ for all $\xi\neq0$, then the multiplier operator $$T_m: e^{i\xi x}\mapsto m(\xi )e^{i\xi x}$$ extends to a 
bounded operator on $L^p({\Bbb R}^n)$ (resp. $L^p({\Bbb T}^n)$) for all $1<p<\infty$. This result was originally proved by S. Mikhlin ([M56], [M65]) and is now  a fundamental theorem  in the  Calder\'on-Zygmund-Stein Singular integral theory. 

Murray and von Neumann's work ([MvN36]) demonstrates   von Neumann algebras  as  a  natural framework to do  noncommutative   analysis. The elements in a von Neumann algerba $\M$ can be ``integrated over" the equipped trace $\tau$ and ``measured" by the associated $L^p$-norms. 
 For a (nonabelian) discrete group $\Gamma$, the von Neumann algebra is the closure of the linear span of left regular representation $\la_g$'s w.r.t. a weak operator topology. The trace $\tau$ is simply defined as 
 $$\tau x=c_e,$$
 for $x=\sum_g c_g\la_g$.
 The associated $L^p$ norm is defined as $$\|x\|_p=(\tau |x|^p)^\frac1p$$ for   $1\leq p<\infty$. When $p=\infty$, the $L^p$ space is set to be  the von Neumann algebra itself. 
 When $\Gamma={\Bbb Z}$, the obtained $L^p$ space is the $L^p$ space on the unit circle ${\Bbb T}=\hat {\Bbb Z}$.
 
 
 The theory of noncommutative $L^p$-spaces was laid out in the work of Segal ([Seg53]) and Dixmier ([Dix53]). 
 Fourier multipliers on noncommutative $L^p$-spaces have been used as fundamental tools in operator algebras theory, noncommutative geometry, and mathematical physics, 
and have grown up to a new studying-object  in functional analysis with its own interest. 
  Fourier multipliers on nonabelian groups  $\Gamma$  are  linear maps $M$ on the left regular representation of $\Gamma$ such that
 $$M:\lambda_g\mapsto m(g)\lambda_g$$ 
 with $m$ a scalar-valued bounded function on $\Gamma$. The boundedness of   $M$ is tested  on   the associated noncommutative $L^p$ spaces.     
 
 The study of Fourier multipliers on free groups  has a long history. Please refer to the works of Bo\.zejko,  Fig\'a-Talamanca/Picardello, Haagerup, Pytlik-Szwarc, Junge/Le Merdy/Xu etc. These works usually rely on the theory of positive definite functions and restrict the study on the radial multipliers. 
 In  [MR17], Mei and Ricard 
 studied an important non-radial multiplier, the so-called free Hilbert transform, and answered a  question of P. Biane and G. Pisier about its  $L^p$ boundedness on free groups.
 
  The key of Mei-Ricard's argument is a free analogue of the classical
  Cotlar's formular. This type of formula fits Hilbert transform  type multipliers but does not work for general Fourier multipliers. This raises the question whether  Mei-Ricard's result is a lucky case or there is  a more general rule on the $L^p$-boundedness of Fourier multipliers.
   In this note, we   study the class of Fourier multipliers on  the free group  associated with the first segment of reduced words, and give a general rule for their complete-$L^p$-boundedness. 
  


\section{A Translation Group}

Given   $z=(z_1,\cdots, z_n)$, a sequence of complex numbers with modular $1$, we  use $T_z,T_z^\circ$ to denote the linear maps  on $L^2 (\hat\F_n)$ such that  $$T_z(\la_e)=\la_e=T^\circ_z(\la_e)$$ and 
$$T_z(\la_h)= z_{i_1}^{k_1}\la_h,\ \ \ T_z^\circ(\la_h)= z_{i_m}^{k_m} \la_h$$ 
for $$h=g_{i_1}^{k_1}g_{i_2}^{k_2}\cdots g_{i_m}^{k_m}.$$ 
Note we have that $$[T_z(\la_h)]^*=T_z^\circ \la_{h^{-1}}.$$
We will prove that $T_z$ is a uniformly bounded group of operators on $L^p(\hat \F_n)$ for all $1<p<\infty$. Therefore, an analogue of the classical Mikhlin's multiplier theorem follows by Coifman/Weiss/Zygmund's transference principle. 

Denote by $\pi_z$ the $*$-homemorphism on ${\cal L}({\F}_n)$ that sends $\la_{g_i}$ to $z_i\la_{g_i}$. Let $P_1$ be the projection onto the subspace of $L^2(\hat\F_n)$ spanned by reduced words  with block length $\leq1$. We see that $$P_1\pi_z=P_1T_z=P_1T_z^\circ, P_1\pi_{z^2}=P_1T_zT_z^\circ,$$ so
\begin{eqnarray}\label{piz}
 \|P_1T_z\|=\|P_1T^\circ_z T_z\|=\|P_1\|\leq 3.
\end{eqnarray}

The following Lemma is from [MR17], Corollary 3.10. One can check the proofs there and find   the upper bound $p^2$ for $p>2$.
\begin{lemma} For $x\in L^p(\hat\F_n)$, we have
\begin{eqnarray}
(\sqrt 2 p^2)^{-1}\|x\|_p\leq \max\{\|\sum_k e_{k1}\otimes L_{k^\pm} x\|_p,\|\sum_k e_{1k}\otimes L_{k^\pm} x\|_p\}\leq  p^2 \|x\|_p.
\end{eqnarray}
for all $2<p<\infty$.
\end{lemma}

\begin{lemma}\label{prop}  For $g,h\in \F_\infty, g\in {\cal L}_{k^\pm},h^{-1}\in {\cal L}_{j^\pm}, k,j\geq0$    we have that\\
 (i) if the block length $\|gh\|\leq1$ and $k=j$,
 \begin{eqnarray} 
 T_z(\la_{g })T_z^\circ(\la_h)= T_z(\la_{g} \la_h)=T_z^\circ(\la_g\la_h); \label{ghid1}
 \end{eqnarray}
(ii) otherwise,
 \begin{eqnarray} T_z(\la_{g })T_z^\circ(\la_h)=T_z(\la_{g } T_z^\circ(\la_h))+T_z^\circ(T_z(\la_{g })\la_h)-T^\circ_z T_z(\la_{g h})\hskip-.1in\label{ghid} 
 \end{eqnarray}
\end{lemma}
{\it Proof.}    In case (i),  suppose $ T_z(\la_{g })=z_k^s\la_g$ and $T_z^o(\la_h)=z_k^l\la_h$, we have $T_z(\la_{g }\la_h)=T^\circ_z(\la_{g h})=z_k^{s+l}\la_{gh}=T_z (\la_g)T^\circ_z(\la_ h)$ since $gh$ has block length $1$. We get (\ref{ghid1}).

In case (ii), we have either  the identity \begin{eqnarray}\label{ei}
T_z(\la_{g}\la_h) =T_z(\la_{g })\la_h
\end{eqnarray}
 or 
 \begin{eqnarray}\label{ther}
 T^\circ_z(\la_{g }\la_h)=  \la_{g }T_z^\circ(\la_h).
 \end{eqnarray}  Assuming (\ref{ei}), we must have    $$T_z(\la_{g }T_z^\circ(\la_h))=T_z(\la_{g })T_z^\circ(\la_h),$$
because $T_z^\circ(\la_h)$ is merely a multiplication of $\la_h$ by a constant.
 We then get (\ref{ghid}).
   Assuming (\ref{ther}),  we have $$T^\circ_z(T_z(\la_{g })\la_h)= T_z( \la_{g })T_z^\circ(\la_h),$$
   because $T_z(\la_g)$ is merely a multiplication of $\la_g$ by a constant.
We get (\ref{ghid}) again.

 \begin{thm}\label{key}
 For  $1<p<\infty$, we have
\begin{eqnarray}\label{key1}  \|T_z  x\|_{L^{p }}\simeq^{ c_{p}}  \| x\|_{L^{p} },\end{eqnarray}
for any $x\in L^p(\hat\F_\infty)$. 
\end{thm}

{\it  Proof.} Assume $\|T_z \|_{L^p(\hat \F_\infty)\rightarrow L^p(\hat \F_\infty)}\leq c_p$ for some $p=2^j,j\geq1$. For $ x=\sum_{g}c_g\la_g$,  denote by $x_k=L_{k^\pm}x$. Let $P_1$ be the projection onto the linear space  corresponding to reduced words with block length smaller or equals to $1$ in $L^2(\F_n)$. Lemma \ref{prop} implies that
\begin{eqnarray}
P_1^\bot[(T_z (x) T^\circ_z( x^*)]\nonumber
&=& P_1^\bot[T_z (xT_z ^\circ(x^*))+T_z ^\circ(T_z (x)x^*)-T^\circ_z T_z (xx^*)] \nonumber\\
P_1 [T_z  (x_k)T^\circ_z (x_j^*)]\nonumber
&=& P_1 [T_z (x_kT_z ^\circ(x_j^*))+T_z ^\circ(T_z (x_k )x^*_j)-T^\circ_z T_z (x_kx^*_j)] \nonumber
\end{eqnarray}
for $k\neq j$, and  
\begin{eqnarray}\label{P1T}
 P_1[T_z(x_k)T_z^\circ(x^*_k)]= P_1[T_z(x_kx^*_k)]=P_1[T_z^\circ(x_kx^*_k)]
\end{eqnarray}
Therefore,
\begin{eqnarray}
T_z (x)T^\circ_z (x^*)\nonumber
&=& [T_z (xT_z ^\circ(x^*))+T_z ^\circ(T_z (x)x^*)-T^\circ_z T_z (xx^*)]\nonumber\\
&-&P_1\sum_k[T_z(  x_kT_z^\circ x^*_k)+T^\circ_z( T_z( x_k)  x^*_k)-T^\circ_z T_z(  x_k  x^*_k)-T_z(  x_k x^*_k)].\nonumber
\end{eqnarray}
Denote by $y=\sum_k e_{k1}\otimes x_k^*$. Note $P_1T_z=P_1T_z^\circ$ and $$P_1T_z(x_kx^*_k)=P_1|T^\circ_z(x^*_k )|^2$$ because of (\ref{P1T}), we have 
\begin{eqnarray}
P_1 T_z(  | y-T^\circ_z y|^2)=P_1\sum_k[T^\circ_z T_z(  x_k  x^*_k)+T_z(  x_k x^*_k)-T_z(  x_kT_z^\circ( x^*_k))-T^\circ_z( T_z( x_k)  x^*_k)].\nonumber
\end{eqnarray}
Therefore,
\begin{eqnarray}\label{cotlar}
T_z (x)T^\circ_z (x^*)\nonumber
&=& [T_z (xT_z ^\circ(x^*))+T_z ^\circ(T_z (x)x^*)-T^\circ_z T_z (xx^*)]\nonumber\\ 
&&+P_1 T_z(  | y-T^\circ_z y|^2). 
\end{eqnarray}
By Lemma 1, we have $$\|y\|_{L^{2p}}=\|T_z y\|_{L^{2p}}\leq (2p)^2\|x\|_{L^{2p}}$$ for $p>1$. Taking $L^p$ norms on both sides of (\ref{cotlar}) and applying (\ref{piz}) and H\"older's inequality,    we get that
\begin{eqnarray*}
\| T_z x\|_{L^{2p}}^2&\leq&2 c_{p}\|x \|_{L^{2p}}\| T_z (x)\|_{L^{2p}}+  (c^2_{p}+192p^4)\|x \|^2_{L^{2p}}.
\end{eqnarray*}
Therefore, $$\|T_z (x)\|_{L^{2p}}\leq (c_p+\sqrt  {2c^2_{p}+192p^4})\| x\|_{L^{2p}}.$$
By induction, we have $$\|T_z x\|_{L^{p}}\leq  8p^2 \| x\|_{L^{p}}$$ for $p=2^n$.
 Applying the fact that $T_z T_{\bar z}=id, \|T_z x\|_{L^2 }\leq \|x\|_{ L^2} $, by interpolation and passing to the dual, we then get  the desired result for all $1<p<\infty$.\qed

\medskip\noindent
{\bf Remark.} Theorem \ref{key} fails for $p=1,\infty$, see the remark after Theorem \ref{main}.
  
  \section{Transference via the Translation Group}
  Given a bounded map $m$ from ${\Bbb Z}$ to ${\Bbb C}$. Let $M_m$ be the linear multiplier on $L^2(\hat{\Bbb F}_n)$  such that
 $$M_m(\la_h)=m(k_1)\la_h$$
 for $h=g_{i_1}^{k_1}g_{i_2}^{k_2}\cdots g_{i_n}^{k_n}$.

\begin{thm}
\label{main}
 For any $1<p<\infty$,    $M_m$ extends to a completely bounded linear operator on $L^p(\hat\F_n)$ iff the restriction of $M_m$ on $\F_1$, denoted by $\tilde M_m$, is  completely bounded Fourier multiplier on $L^p(\hat{\Bbb F}_1)=L^p({\Bbb T})$. Moreover,
 \begin{eqnarray*}
 \|M_m\|\leq c^2_p \|\tilde M_m\| 
 \end{eqnarray*}
 with $c_p$ the equivalence constant in Theorem \ref{key}.
  \end{thm}
{\it Proof.}
By Theorem \ref{key}, we have $$T_t:\la_h\mapsto e^{itk_1}\la_h$$ for $h=g_{i_1}^{k_1}g_{i_2}^{k_2}\cdots g_{i_n}^{k_n}$
is a  uniformly bounded $c_0$-group of operators on $L^p$. The desired result follows by Coifman/Weiss/Zygmund's transference principle.
Assume $\tilde M_m$ extends to a  completely bounded Fourier multiplier on $ L^p({\Bbb T})$. By approximation,  we may assume that $m$ has a finite support $[-N,  N]$. Define the scalar valued function $\phi_N$ on the unit circle as $\phi_N=\sum_k m(k)e^{ik\theta}$. Then 
we have, for any $L^p(\hat\F_n)$-valued function $F$, $$\|F*\phi_N\|_{L^p({\Bbb T},L^p(\hat\F_n))}=\|\tilde M_m(F)\|_{L^p({\Bbb T},L^p(\hat\F_n))}\leq \|\tilde M_m\| \|F\|_{L^p({\Bbb T},L^p(\hat\F_n))}.$$
 For any $x\in L^p(\hat\F_n)$ with norm $1$, we have 
\begin{eqnarray*}
M_m(x)&=&\frac1{2\pi}\int_0^{2\pi} T_{-t}x\phi_N(t)dt\\
&=&\frac1{2\pi}\int_0^{2\pi} T_{-s}T_{s-t}x\phi_N(t)dt.
\end{eqnarray*}
 Let $F(t)=T_tx$. Then $\|F\|_{L^p([0,2\pi],L^p(\hat\F_n))}\leq c_p$ and
\begin{eqnarray*}
 \|M_m(x)\|^p_{ L^p(\hat\F_n)}&\leq c_p &\frac1{2\pi}\int_0^{2\pi}\|\frac1{2\pi}\int_0^{2\pi} T_{s-t}x\phi_N(t)dt\|^p_{L^p(\hat\F_n)}ds \\
&=& c_p\|F*\phi_N\|^p_{L^p({\Bbb T},L^p(\hat\F_n))} \\
&\leq& c_p\|\tilde M_m\|\|F\|_{L^p({\Bbb T},L^p(\hat\F_n))}\leq c_p^2\|\hat M_m\|.
 \end{eqnarray*}\qed
 
\noindent
{\bf Remark.} Theorem \ref{main} fails for $p=1,\infty$. Take $m(k)=\chi_{[-2,2]}(k)$, which is the symbol of a c.b multiplier on $L^\infty(\hat\F_1)$. Then the multiplier $M_m$ is the projection   onto the set $\{s; |k_1(s)|\leq 2\}$ of $L^\infty(\hat\F_\infty)$, which is certainly unbounded. To see this, let
$$x=\sum_{-N<k<N} c_k (\la_{g_1g_2^3})^kg_1$$ and note $$M_m(x)=\sum_{0\leq k<N}  c_k (\la_{g_1g_2^3})^kg_1.$$

Theorem \ref{key} fails for $p=1,\infty$ too, since  Theorem \ref{main} would follow from it.
 
 \medskip
 
\noindent Let  $A_0=\{0\}, A_k= [2^{k-1},2^k)$ for $k\in {\Bbb N}$ and $ A_k= -A_{-k}$ for $k\in -{\Bbb N}$. 

 \medskip
 \begin{corollary} Let $\chi_k$ be the characteristic function on $ A_k$ for $k\in {\Bbb Z}$. Then
\begin{eqnarray}
\|x\|_{L^p}\simeq \|(M_{\chi_k}(x))_k\|_{L^p(\ell_{cr}^2)}.
\end{eqnarray}
\end{corollary}
\begin {proof}  
Let $m =\sum_{k\in {\Bbb Z}} \varepsilon_k \chi_k$   with $\varepsilon_k=\pm$. Then $M_m$ is a completely bounded multiplier on $L^p({\Bbb T})$. The Khitchine inequality and Theorem above imply that \begin{eqnarray}
\|(M_{\psi_k}(x)\|_{L^p(\ell_{cr}^2)}\lesssim   \|x\|_{L^p}. 
\end{eqnarray}
The other direction of the inequality follows from the duality between $L^p$ spaces and the identity 
  \begin{eqnarray}
\langle x, y\rangle= \sum_k \langle M_{\chi_k}(x) ,M_{\chi_k}(y) \rangle.
\end{eqnarray}
\end{proof}

\begin{corollary} \label{Mik} Suppose $M$ is a Mikhlin multiplier in the sense that  
$$ \sup_{k\in {\Bbb Z}} \{|m(k)|, k|m(k)-m(k-1)|\}<C.$$  Then $M$ extends to a completely bounded linear operator on $L^p(\hat\F_n)$ for all $1<p<\infty$.
\end{corollary}
\begin{proof} This follows from Theorem \ref{main} and the classical Mikhlin multiplier theorem.
\end{proof}

\begin{corollary} Let $\psi$ be a $C^2$ function supported on $[\frac12,\frac32]$ and $\psi(t)=1$ for $t\in [\frac23,\frac43]$. Let $\psi_k(t)=\psi(\frac t{2^k})$. Then
\begin{eqnarray}
\|x\|_{L^p}\simeq \|(M_{\psi_k}(x))_k\|_{L^p(\ell_{cr}^2)}.
\end{eqnarray}
\end{corollary}
\begin{proof} Let $m_1=\sum_{k\ odd} \varepsilon_k \psi_k$ and $m_2=\sum_{k\ even} \varepsilon_k\psi_k$ with $\varepsilon_k=\pm$. The Khitchine inequality and  Corollary  \ref{Mik} imply that \begin{eqnarray}
\|(M_{\psi_k}(x)\|_{L^p(\ell_{cr}^2)}\lesssim   \|x\|_{L^p}.
\end{eqnarray}
The other direction of the inequality follows from the duality between $L^p$ spaces and the identity 
  \begin{eqnarray}
\langle x, y\rangle= \sum_k \langle M_{\psi_k}(x) ,M_{\chi_k}(y) \rangle.
\end{eqnarray}
\end{proof}

\begin{corollary} Let $1 < p < \infty$. Then the unbounded linear operator $L:\la_h\mapsto k_1(h)\la_h$ has a bounded $H^\infty$-functional calculus on $L^p(\hat\Gamma)$ of
any positive angle $\mu$  Moreover, we have that
\begin{eqnarray}
\|\Phi(L)\| \lesssim (\sin\mu)^{-1}\|\Phi\|_\infty,
\end{eqnarray}
for all $\Phi\in H^\infty(\Sigma_\mu)$.
\end{corollary}
\begin{proof} Applying Cauchy's integral formula, it is easy to check that 
$$ \sup_{t\in {\Bbb R}} \{|\Phi(t)|, t|\Phi'(t)|\}<C(\sin\mu)^{-1}\|\Phi\|_\infty $$ for   $\Phi\in H^\infty(\Sigma_\mu)$. The desired result follows from Corollary \ref{Mik}.
\end{proof}

\medskip
Recall that we say a subset $\Lambda$ of $\Bbb Z$ is a c. b. $\Lambda_p (p>2)$ set with constant $C_\Lambda$ if,  
for any operator valued sequence $x_k, k\in \Lambda$, we have the equivalence
\begin{eqnarray*}
\|\sum _{k\in \Lambda} x_k e^{ik\theta}\|\simeq^{ C_\Lambda}\|(x_k)_{k\in \Lambda}\|_{L^p(\ell_{cr}^2)}.
\end {eqnarray*}
This is equivalent to   that, for  any subset $A\subset \Lambda$,  the Fourier multiplier $$\tilde M_{\chi_A}: e^{ik\theta}\mapsto \chi_A(k)e^{ik\theta}$$ extends to a completely bounded map on $L^p(\hat {\Bbb Z})=L^p({\Bbb T})$ with a bound $\leq C_\Lambda$.


\begin{corollary} Suppose $\Lambda\subset\Bbb Z$ is a c. b. $\Lambda_p$ set. Then, for any $A\subset \Lambda $,  $M_{\chi_A}$ extends to a completely bounded map on $L^p(\hat\F_n)$ with a bound $\leq c_p^2C_\Lambda$.
\end{corollary}
\begin{proof} This follows from Theorem \ref{main}.
\end{proof}
 \bibliographystyle{amsplain}

\begin{thebibliography}{99}

 
 \bibitem[BF06]{BF06}  M. Bo\.ejko, G. Fendler, 
A note on certain partial sum operators. (English summary) Quantum probability, 117-125, 
Banach Center Publ., 73, Polish Acad. Sci. Inst. Math., Warsaw, 2006.

 

 \bibitem[Bu99]{Bu99} A. Buchholu, Norm of convolution by operator-valued functions on free groups,
 Proc. Amer. Math. Soc. 127 (1999), 1671-1682. 
 
 

 \bibitem[CW76] {CW76} R. Coifman, G. Weiss,
Transference methods in analysis. 
Conference Board of the Mathematical Sciences Regional Conference Series in Mathematics, No. 31. American Mathematical Society, Providence, R.I., 1976. ii+59 pp.
 
 
 \bibitem [Dix53] {Dix53} J. Dixmier, Formes lin\'eaires sur un anneau d'op\'erateurs. (French) Bull. Soc. Math. France 81, (1953). 9-39.
 
 \bibitem[HP98]  {HP98} M. Hieber, J. Pr\"uss, Funtional calculi for linear operators in vector valued $L^p$-spaces via the transference principle, Adv. Differential Equations, 3 (1998), 847-872.

 \bibitem[HSS10]{HSS10}  U. Haagerup, T. Steenstrup, R. Szwarc,  
Schur multipliers and spherical functions on homogeneous trees.  
Internat. J. Math. 21 (2010), no. 10, 1337-1382. 

 \bibitem [JMP18] {JMP18} M. Junge, T. Mei and J. Parcet, Noncommutative Riesz transforms- a dimension free estimate, Journal of European Math.  Soc.  (JEMS) 20 (2018),no.  3, 529-595.
   
  \bibitem[MvN36]{MvN36}  F. J. Murray,  J. von Neumann,   On rings of operators, Annals of Mathematics, Second Series, 37 (1936), 116-229.
   
   \bibitem [M56]  {M56} S. Mikhlin, On the multipliers of Fourier integrals", Doklady Akademii Nauk SSSR, 109: 701-703.
   
 \bibitem[M65] {M65} S. Mikhlin,   Multidimensional singular integrals and integral equations, International Series of Monographs in Pure and Applied Mathematics, 83, Pergamon Press, Zbl 0129.07701.

  \bibitem[MR17] {MR17} T. Mei, E. Ricard, Free Hilbert Transforms, Duke Journal of Math., 166, 2232-2250.
   
    \bibitem[MS17] {MS17} T. Mei, M. de la Salle, Mikael, Complete boundedness of heat semigroups on the von Neumann algebra of hyperbolic groups. Trans. Amer. Math. Soc. 369 (2017), no. 8, 5601-5622.
 
   \bibitem[O10] {O10} N. Ozawa, 
A comment on free group factors. (English summary) Noncommutative harmonic analysis with applications to probability II, 241-245,
Banach Center Publ., 89, Polish Acad. Sci. Inst. Math., Warsaw, 2010.

  

 \bibitem[RX06]{RX06}  E. Ricard, Q. Xu,  Khintchine type inequalities for reduced free products and applications. J. Reine. Angew. Math. 599 (2006), 27--59.
 
 \bibitem[Seg53]{Seg53} I. Segal,   A non-commutative extension of abstract integration. Ann. of Math. (2) 57, (1953). 401-457.
 
 \end{thebibliography}

\bigskip
\hfill \noindent \textbf{Tao Mei} \\
\null \hfill Department of Mathematics
\\ \null \hfill Baylor University \\
\null \hfill One bear place, Waco, TX  USA \\
\null \hfill\texttt{tao\_mei@baylor.edu}

\bigskip
\hfill \noindent \textbf{Quanhua Xu} \\
\null \hfill Institute for Advanced Study in Mathematics\\
\null \hfill  Harbin Institute of Technology\\
\null \hfill  Harbin 150001 HEILO\\
\null \hfill PEOPLES REPUBLIC OF CHINA\\

\null \hfill Laboratoire de Math\'ematiques
\\ \null \hfill Univ. de Franch Comt\'e \\
\null \hfill 25030, Besan\c on, France \\
\null \hfill\texttt{quanhua.xu@univ-fcomte.fr}

\end{document}